\documentclass{article}

\usepackage{amsthm,amssymb,amsmath}
\usepackage{epsfig}
\usepackage{color}
\usepackage{verbatim}

\newtheorem{prelem}{{\bf Theorem}}

\newtheorem{theorem}{Theorem}
\newtheorem{corollary}[theorem]{Corollary}
\newtheorem{lemma}[theorem]{Lemma}
\newtheorem{example}[theorem]{Example}

\newtheorem{proposition}[theorem]{Proposition}

\theoremstyle{definition}
\newtheorem{definition}[theorem]{Definition}
\theoremstyle{remark}
\newtheorem{rem}[theorem]{\bf Remark}

\let\Bbb=\mathbb
\let\phi=\varphi
\def\NZQ{\Bbb}

\def\ZZ{{\NZQ Z}}
\def\QQ{{\NZQ Q}}
\def\RR{{\NZQ R}}
\def\NN{{\NZQ N}}

\def\relint{\operatorname{relint}}

\let\oldbigwedge\bigwedge
\def\BIGwedge{{\textstyle\oldbigwedge}}
\def\medwedge{{\scriptstyle\oldbigwedge}}
\def\bigwedge{\mathchoice{\BIGwedge}{\BIGwedge}{\medwedge}{}}
\def\dim{\operatorname{dim}}

\let\epsilon=\varepsilon

\title{Cohen-Macaulay types of certain edge subrings of bipartite graphs and (generalized) Fuss-Catalan numbers}
\author{Alin \c{S}tefan}
\date{}
\begin{document}
\maketitle

\begin{abstract}

We give an example of two non isomorphic coordinate rings of a special kind of convex polyominoes whose Cohen-Macaulay types are generalized Fuss-Catalan numbers. We further provide a determinantal formula for these numbers.
\end{abstract}


\section{Introduction}
Polyominoes are two-dimensional objects obtained by joining unitary squares edge by edge. They appear in puzzles, designs, and tiling problems (see $\cite{G}$ for more details).

Let $\mathcal{P}$ be a polyomino and $K$ be a field. We denote by $I_{\mathcal{P}}$, the polyomino ideal attached to $\mathcal{P}$, in a suitable polynomial ring over $K$. The associated quotient ring is denoted by $K[\mathcal{P}]$.
In $\cite{Q}$, when $\mathcal{P}$ is convex, Qureshi proved that $K[\mathcal{P}]$ is a normal domain. She proved that $I_{\mathcal{P}}$ defines the edge subring $R=K[x_{i}y_{j} \ | \ (i,j)\in V(\mathcal{P})]$, 
where $V(\mathcal{P})$ is the set of all corners belonging to cells of $\mathcal{P}$.
 In this case, $R$ is Cohen-Macaulay by a celebrated theorem of Hochster. Danilov and Stanley described the canonical module $\omega_{R}$ of $R$ via the polyhedral cone generated by the set of exponents $\{\log(x_{i}y_{j}) \ | \ (i,j)\in V(\mathcal{P})\}.$
The Cohen-Macaulay type of $R$ is defined by $r(R)=\mu(R)=\dim_{K} \ \omega_{R} \otimes K.$
The canonical module is a fundamental object in commutative algebra that contains information about the last syzygy module of $K[\mathcal{P}]$ (see $\cite{BH}$,  pp. 136-146).
In $\cite{VV}$ the authors show how to compute a generating set for the canonical module of the edge ring of bipartite graph using
combinatorial optimization tehniques in terms of the incidence matrix of graph and the vertices of a certain blocking polyhedron, but an explicit formula for the Cohen-Macaulay type is not given. 

Let $\mathbb{N}$ denote the set of nonnegative integers, $[m]$ the set $\{1, 2, \ldots, m\}$ when $m\in \mathbb{N}$ and $\ZZ_>$ the set of positive integers. 
For $p \geq 1$ let $u_{1}, \ldots, u_{p}$, $r_{1}, \ldots, r_{p} \in \ZZ_>$ and for $k\in [p]$ we write \[A_{k}=1 + \sum_{i=1}^{k}u_{i}, \ B_{k}=1 + \sum_{i=1}^{k}r_{i}.\]

Let $\mathcal{P}=\mathcal{P}
  \left( {\begin{array}{cccc}
   u_{1} & u_{2} & . . . & u_{p} \\
   r_{1} & r_{2} & . . . & r_{p} \\
  \end{array} } \right)
$ 
be the polyomino as shown in Fig. 1. This has the first $r_{1}$ columns of height $A_{1}-1$, next $r_{2}$ columns of height $A_{2}-1$ and so on. The coordinate ring $K[\mathcal{P}]$ is 
a normal Cohen-Macaulay domain of dimension $|V(\mathcal{P})|$ - $|\mathcal{P}|$=$A_{p}+B_{p}-1$. Also, $K[\mathcal{P}]$ is isomorphic to the toric ring $R=K[x_{i}y_{j} \ | \ (i,j)\in V(\mathcal{P})]$.
We present the irreducible representation of the cone generated by \[B=\{\log(x_{i}y_{j}) \ | \ (i,j)\in V(\mathcal{P})\}\subset \NN^{B_{p}+A_{p}},\] 
the set of the exponents of the generators of the toric ring $R$.

\unitlength 0.67mm 
\linethickness{0.4pt}
\ifx\plotpoint\undefined\newsavebox{\plotpoint}\fi 

\begin{picture}(20,120)(-30,0)
\put(40,40){\framebox(4,4)[lc]{}}
\put(44,40){\framebox(4,4)[lc]{}}
\put(48,40){\framebox(4,4)[lc]{}}
\put(52,40){\framebox(4,4)[lc]{}}
\put(56,40){\framebox(4,4)[lc]{}}
\put(60,40){\framebox(4,4)[lc]{}}
\put(64,40){\framebox(4,4)[lc]{}}
\put(68,40){\framebox(4,4)[lc]{}}
\put(72,40){\framebox(4,4)[lc]{}}
\put(76,40){\framebox(4,4)[lc]{}}
\put(80,40){\framebox(4,4)[lc]{}}
\put(84,40){\framebox(4,4)[lc]{}}
\put(88,40){\framebox(4,4)[lc]{}}
\put(92,40){\framebox(4,4)[lc]{}}
\put(96,40){\framebox(4,4)[lc]{}}
\put(100,40){\framebox(4,4)[lc]{}}
\put(104,40){\framebox(4,4)[lc]{}}

\put(19,40){$M_{0}(1,1)$}

\put(40,44){\framebox(4,4)[lc]{}}
\put(44,44){\framebox(4,4)[lc]{}}
\put(48,44){\framebox(4,4)[lc]{}}
\put(52,44){\framebox(4,4)[lc]{}}
\put(56,44){\framebox(4,4)[lc]{}}
\put(60,44){\framebox(4,4)[lc]{}}
\put(64,44){\framebox(4,4)[lc]{}}
\put(68,44){\framebox(4,4)[lc]{}}
\put(72,44){\framebox(4,4)[lc]{}}
\put(76,44){\framebox(4,4)[lc]{}}
\put(80,44){\framebox(4,4)[lc]{}}
\put(84,44){\framebox(4,4)[lc]{}}
\put(88,44){\framebox(4,4)[lc]{}}
\put(92,44){\framebox(4,4)[lc]{}}
\put(96,44){\framebox(4,4)[lc]{}}
\put(100,44){\framebox(4,4)[lc]{}}
\put(104,44){\framebox(4,4)[lc]{}}

\put(52,48){\framebox(4,4)[lc]{}}
\put(56,48){\framebox(4,4)[lc]{}}
\put(60,48){\framebox(4,4)[lc]{}}
\put(64,48){\framebox(4,4)[lc]{}}
\put(68,48){\framebox(4,4)[lc]{}}
\put(72,48){\framebox(4,4)[lc]{}}
\put(76,48){\framebox(4,4)[lc]{}}
\put(80,48){\framebox(4,4)[lc]{}}
\put(84,48){\framebox(4,4)[lc]{}}
\put(88,48){\framebox(4,4)[lc]{}}
\put(92,48){\framebox(4,4)[lc]{}}
\put(96,48){\framebox(4,4)[lc]{}}
\put(100,48){\framebox(4,4)[lc]{}}
\put(104,48){\framebox(4,4)[lc]{}}

\put(16,48){$M_{1}(1, A_{1})$}

\put(52,52){\framebox(4,4)[lc]{}}
\put(56,52){\framebox(4,4)[lc]{}}
\put(60,52){\framebox(4,4)[lc]{}}
\put(64,52){\framebox(4,4)[lc]{}}
\put(68,52){\framebox(4,4)[lc]{}}
\put(72,52){\framebox(4,4)[lc]{}}
\put(76,52){\framebox(4,4)[lc]{}}
\put(80,52){\framebox(4,4)[lc]{}}
\put(84,52){\framebox(4,4)[lc]{}}
\put(88,52){\framebox(4,4)[lc]{}}
\put(92,52){\framebox(4,4)[lc]{}}
\put(96,52){\framebox(4,4)[lc]{}}
\put(100,52){\framebox(4,4)[lc]{}}
\put(104,52){\framebox(4,4)[lc]{}}

\put(52,56){\framebox(4,4)[lc]{}}
\put(56,56){\framebox(4,4)[lc]{}}
\put(60,56){\framebox(4,4)[lc]{}}
\put(64,56){\framebox(4,4)[lc]{}}
\put(68,56){\framebox(4,4)[lc]{}}
\put(72,56){\framebox(4,4)[lc]{}}
\put(76,56){\framebox(4,4)[lc]{}}
\put(80,56){\framebox(4,4)[lc]{}}
\put(84,56){\framebox(4,4)[lc]{}}
\put(88,56){\framebox(4,4)[lc]{}}
\put(92,56){\framebox(4,4)[lc]{}}
\put(96,56){\framebox(4,4)[lc]{}}
\put(100,56){\framebox(4,4)[lc]{}}
\put(104,56){\framebox(4,4)[lc]{}}

\put(52,60){\framebox(4,4)[lc]{}}
\put(56,60){\framebox(4,4)[lc]{}}
\put(60,60){\framebox(4,4)[lc]{}}
\put(64,60){\framebox(4,4)[lc]{}}
\put(68,60){\framebox(4,4)[lc]{}}
\put(72,60){\framebox(4,4)[lc]{}}
\put(76,60){\framebox(4,4)[lc]{}}
\put(80,60){\framebox(4,4)[lc]{}}
\put(84,60){\framebox(4,4)[lc]{}}
\put(88,60){\framebox(4,4)[lc]{}}
\put(92,60){\framebox(4,4)[lc]{}}
\put(96,60){\framebox(4,4)[lc]{}}
\put(100,60){\framebox(4,4)[lc]{}}
\put(104,60){\framebox(4,4)[lc]{}}

\put(60,64){\framebox(4,4)[lc]{}}
\put(64,64){\framebox(4,4)[lc]{}}
\put(68,64){\framebox(4,4)[lc]{}}
\put(72,64){\framebox(4,4)[lc]{}}
\put(76,64){\framebox(4,4)[lc]{}}
\put(80,64){\framebox(4,4)[lc]{}}
\put(84,64){\framebox(4,4)[lc]{}}
\put(88,64){\framebox(4,4)[lc]{}}
\put(92,64){\framebox(4,4)[lc]{}}
\put(96,64){\framebox(4,4)[lc]{}}
\put(100,64){\framebox(4,4)[lc]{}}
\put(104,64){\framebox(4,4)[lc]{}}

\put(60,68){\framebox(4,4)[lc]{}}
\put(64,68){\framebox(4,4)[lc]{}}
\put(68,68){\framebox(4,4)[lc]{}}
\put(72,68){\framebox(4,4)[lc]{}}
\put(76,68){\framebox(4,4)[lc]{}}
\put(80,68){\framebox(4,4)[lc]{}}
\put(84,68){\framebox(4,4)[lc]{}}
\put(88,68){\framebox(4,4)[lc]{}}
\put(92,68){\framebox(4,4)[lc]{}}
\put(96,68){\framebox(4,4)[lc]{}}
\put(100,68){\framebox(4,4)[lc]{}}
\put(104,68){\framebox(4,4)[lc]{}}

\put(60,72){\framebox(4,4)[lc]{}}
\put(64,72){\framebox(4,4)[lc]{}}
\put(68,72){\framebox(4,4)[lc]{}}
\put(72,72){\framebox(4,4)[lc]{}}
\put(76,72){\framebox(4,4)[lc]{}}
\put(80,72){\framebox(4,4)[lc]{}}
\put(84,72){\framebox(4,4)[lc]{}}
\put(88,72){\framebox(4,4)[lc]{}}
\put(92,72){\framebox(4,4)[lc]{}}
\put(96,72){\framebox(4,4)[lc]{}}
\put(100,72){\framebox(4,4)[lc]{}}
\put(104,72){\framebox(4,4)[lc]{}}

\put(24,64){$M_{2}(B_{1}, A_{2})$}

\put(72,76){\framebox(4,4)[lc]{}}
\put(76,76){\framebox(4,4)[lc]{}}
\put(80,76){\framebox(4,4)[lc]{}}
\put(84,76){\framebox(4,4)[lc]{}}
\put(88,76){\framebox(4,4)[lc]{}}
\put(92,76){\framebox(4,4)[lc]{}}
\put(96,76){\framebox(4,4)[lc]{}}
\put(100,76){\framebox(4,4)[lc]{}}
\put(104,76){\framebox(4,4)[lc]{}}

\put(72,80){\framebox(4,4)[lc]{}}
\put(76,80){\framebox(4,4)[lc]{}}
\put(80,80){\framebox(4,4)[lc]{}}
\put(84,80){\framebox(4,4)[lc]{}}
\put(88,80){\framebox(4,4)[lc]{}}
\put(92,80){\framebox(4,4)[lc]{}}
\put(96,80){\framebox(4,4)[lc]{}}
\put(100,80){\framebox(4,4)[lc]{}}
\put(104,80){\framebox(4,4)[lc]{}}

\put(88,84){\framebox(4,4)[lc]{}}
\put(92,84){\framebox(4,4)[lc]{}}
\put(96,84){\framebox(4,4)[lc]{}}
\put(100,84){\framebox(4,4)[lc]{}}
\put(104,84){\framebox(4,4)[lc]{}}

\put(88,88){\framebox(4,4)[lc]{}}
\put(92,88){\framebox(4,4)[lc]{}}
\put(96,88){\framebox(4,4)[lc]{}}
\put(100,88){\framebox(4,4)[lc]{}}
\put(104,88){\framebox(4,4)[lc]{}}

\put(88,92){\framebox(4,4)[lc]{}}
\put(92,92){\framebox(4,4)[lc]{}}
\put(96,92){\framebox(4,4)[lc]{}}
\put(100,92){\framebox(4,4)[lc]{}}
\put(104,92){\framebox(4,4)[lc]{}}

\put(96,96){\framebox(4,4)[lc]{}}
\put(100,96){\framebox(4,4)[lc]{}}
\put(104,96){\framebox(4,4)[lc]{}}

\put(28,84){$M_{p-2}(B_{p-3}, A_{p-2})$}

\put(96,100){\framebox(4,4)[lc]{}}
\put(100,100){\framebox(4,4)[lc]{}}
\put(104,100){\framebox(4,4)[lc]{}}

\put(44,96){$M_{p-1}(B_{p-2}, A_{p-1})$}

\put(96,104){\framebox(4,4)[lc]{}}
\put(100,104){\framebox(4,4)[lc]{}}
\put(104,104){\framebox(4,4)[lc]{}}

\put(63,112){$M_{p}(B_{p-1}, A_{p})$}

\put(96,108){\framebox(4,4)[lc]{}}
\put(100,108){\framebox(4,4)[lc]{}}
\put(104,108){\framebox(4,4)[lc]{}}

\put(110,112){$M_{p+1}(B_{p}, A_{p})$}

\put(-38,27){Fig. 1 The polyomino 
$\mathcal{P}
  \left( {\begin{array}{cccc}
   u_{1} & u_{2} & . . . & u_{p} \\
   r_{1} & r_{2} & . . . & r_{p} \\
  \end{array} } \right)
$ }
\end{picture}\\
If $\mathcal{P}_{1}=\mathcal{P}
  \left( {\begin{array}{cccc}
   u_{1} & u_{2} & . . . & u_{p} \\
   r_{1} & r_{2} & . . . & r_{p} \\
  \end{array} } \right)
	$
	and
	$\mathcal{P}_{2}=\mathcal{P}
  \left( {\begin{array}{cccc}
   u_{1} & u_{2} & . . . & u_{p} \\
   s_{1} & s_{2} & . . . & s_{p} \\
  \end{array} } \right)
	$
	where $u_{1}=\ldots=u_{n}=n, \ r_{1}=\ldots=r_{n}=t$ and $s_{1}=\ldots=s_{n}=n-t$,  we show that two non isomorphic coordinate rings $K[\mathcal{P}_{1}]$ and $K[\mathcal{P}_{2}]$ have the same Cohen-Macaulay types 
	$r(K[\mathcal{P}_{1}])=\begin{bmatrix}  n \\ t \end{bmatrix}_{p} = \begin{bmatrix}  n \\ n-t \end{bmatrix}_{p} =r(K[\mathcal{P}_{2}])$  which are the generalized Fuss-Catalan numbers.


\section{Prelimininaries}
\subsection{Cones}

In this section we fix the notation and recall some basic results. For details we refer the reader to
 \cite{B}, \cite{BG}, \cite{BH}, \cite{MS},  \cite{V} and \cite{We}.

We denote $\QQ_+,\RR_+$ the set of nonnegative rationals and real numbers respectively.

Fix an integer $n > 0$.  If $0\neq a\in \mathbb{Q}^{n},$ then $H_{a}$ will denote the
rational hyperplane of $\mathbb{R}^{n}$ through the origin with normal vector
$a,$ that is,\[H_{a}=\{x\in \mathbb{R}^{n}\ | \ \langle x,a\rangle=0\},\] where
$\langle \ ,\ \rangle$ is the scalar product in $\mathbb{R}^{n}.$ The two closed rational linear
halfspaces bounded by $H_{a}$ are: \[H^{+}_{a}=\{x\in
\mathbb{R}^{n}\ | \ \langle x,a\rangle \geq 0\} \ {\rm{and}} \  H^{-}_{a}=H^{+}_{-a}=\{x\in
\mathbb{R}^{n}\ | \ \langle x,a\rangle \leq  0\}.\]
The two open rational linear
halfspaces bounded by $H_{a}$ are: \[H^{>}_{a}=\{x\in
\mathbb{R}^{n}\ | \ \langle x,a\rangle > 0\} \ {\rm{and}} \  H^{<}_{a}=H^{>}_{-a}=\{x\in
\mathbb{R}^{n}\ | \ \langle x,a\rangle <  0\}.\]

If $S\subset \QQ^n$, then the set
$$\mathbb{R}_{+}{S}=\{\sum_{i=1}^{r}a_{i}v_{i}\ : \ a_{i}\in
\mathbb{R}_{+}, \ v_i \in S, \ r\in\NN \}$$
is called the \emph{rational cone} generated by $S$. The \emph{dimension} of a cone is the dimension of the smallest vector subspace of $\RR^n$ which contains it.

By the theorem of Minkowski-Weyl, see \cite{BG}, \cite{H}, \cite{We}, finitely generated rational
cones can also be described as intersection of finitely many rational closed subspaces (of the form
$H^{+}_{a}$). We further restrict this presentation to the class of finitely generated rational cones, which will be simply called cones.
If a cone $C$ is presented as \[C=H^{+}_{a_{1}}\cap \ldots
\cap H^{+}_{a_{r}}\] such that no  $H^{+}_{a_{i}}$ can be
omitted, then we say that this is an \emph{irredundant representation} of $C$. If $\dim(C)=n$, then the  halfspaces
$H^{+}_{a_{1}}, \ldots , H^{+}_{a_{r}}$
in an irredundant representation of $C$
are uniquely determined and we set
\[\relint(C)=H^{>}_{a_{1}}\cap \ldots
\cap H^{>}_{a_{r}}\]
the \emph{relative interior} of $C$.
If $a_{i}=(a_{i1},\ldots,a_{in}),$  then we call
\[H_{a_{i}}(x):=a_{i1}x_{1}+\cdots + a_{in}x_{n}=0,
\]
the \emph{equations of the cone} $C.$

A hyperplane $H$ is called a supporting hyperplane of a cone $C$ if $C\cap H\neq \emptyset$ and $C$ is contained in one of the closed halfspaces determined by $H$.
If $H$ is a supporting hyperplane of $C$, then $F=C\cap H$ is called a \emph{proper face} of $C$. It is convenient to consider also the empty set and $C$ as faces, the \emph{improper faces}.
The faces of a cone are themselves cones. A face $F$ with $\dim(F)=\dim(C)-1$ is called a \emph{facet}. If $\dim \ \mathbb{R}_{+}S=n$ and $F$ is a facet defined by
the supporting hyperplane $H$, then $H$ is generated as a linear subspace by a linearly independent subset of $S$.

A cone $C$ is  \emph{pointed} if $0$ is a face of $C$. This is equivalent to say that $x\in C$ and $-x\in C$ imply $x=0$.
The faces of dimension $1$ of a pointed cone are called \emph{extreme rays}.

Let $\mathcal{A}\subset \mathbb{Z}^{n}$ be a collection of lattice points in $\mathbb{Z}^{n}$ and $R=K[\mathcal{A}]$ be the toric ring defined by $\mathcal{A}$ over field $K.$
By a celebrated theorem of Hochster, when $R$ is normal, it is Cohen-Macaulay. In this case, Danilov and Stanley $( \cite{BH},$ Theorem \ 6.3.5 and $\cite{D})$ described
the canonical module $\omega_{R}$ of $R$ via the polyhedral cone generated by $\mathcal{A},$ that is
$\omega_{R}=(\{x^{a}| \ a\in {\mathbb{N}}\mathcal{A}\cap
 \relint({\mathbb{R_{+}}}\mathcal{A})\}).$  The Cohen-Macaulay type of $R$ is defined by $r(R)=\mu(\omega_{R})=\dim_{K} \ \omega_{R} \otimes K.$

\subsection{Polyominoes}

First, we briefly recall fundamental materials and basic terminologies on polyominoes and their binomial ideals.
For further information on algebra and combinatorics on polyominoes we refer  the reader to $\cite{Q}$.

Given $a=(i,j)$ and $b=(k,l)$ belonging to $\mathbb{N}^{2}$ we write $a<b$ if $i<k$ and $j<l.$ When $a<b$, we define 
an $interval$ $[a,b]$ of $\mathbb{N}^{2}$ to be $[a,b]=\{c\in \mathbb{N}^{2} : a\leq c\leq b\}\subset \mathbb{N}^{2}.$
For an interval $[a,b]$, the $diagonal$ corners of $[a,b]$ are $a$ and $b$ and the $anti$-$diagonal$ corners of $[a,b]$ are
$c=(i,l)$ and $d=(k,j)$.

A $cell$ of $\mathbb{N}^{2}$ with the lower left corner $a\in \mathbb{N}^{2}$ is the interval $C=[a,a+(1,1)]$. Its $vertices$ 
are $a$, $a+(1,1)$, $a+(1,0)$ and $a+(0,1)$. Its $edges$ are $\{a, a+(1,0)\}$, $\{a, a+(0,1)\}$, $\{a+(1,0), a+(1,1)\}$, 
$\{a+(0,1), a+(1,1)\}$. Let $V(C)$ denote the set of vertices of $C$ and $E(C)$ the set of edges of $C$. 

Let $\mathcal{P}$ be a finite collection of cells of $\mathbb{N}^{2}$. Its $vertex \ set$ is $V(\mathcal{P})=
\bigcup_{C\in \mathcal{P}}V(C)$ and its $edge \ set$ is $E(\mathcal{P})=\bigcup_{C\in \mathcal{P}}E(C)$. We say that
two cells $C$ and $D$ are $connected$ if there exists a sequence of cells \[\mathcal{C}: C=C_{1},\ldots,C_{n}=D\] of 
$\mathcal{P}$ such that $C_{i}\cup C_{i+1}$ is an edge of $C_{i}$ for $i\in [n-1]$; if $C_{i}\neq C_{j}$ for all 
$i\neq j$, then $\mathcal{C}$ is called a $path$. We say that $\mathcal{P}$ is a  $polyomino$ if any two cells of 
$\mathcal{P}$ are connected. A polyomino $\mathcal{Q}$ is a $subpolyomino$ of $\mathcal{P}$ and denoted $\mathcal{Q}\subset \mathcal{P}$ 
if each cell belonging to $\mathcal{Q}$ belongs to $\mathcal{P}$.

Let $A$ and $B$ be cells of $\mathbb{N}^{2}$ for which $(i,j)$ is the lower left corner of $A$ and $(k,l)$ is the lower left corner of $B$. 
If $i\leq k$ and $j\leq l$, then the $cell \ interval$ of $A$ and $B$ is the set $[A,B]$ which consists of those 
cells $E$ of $\mathbb{N}^{2}$ whose lower left corner $(r,s)$ satisfies $i\leq r\leq k$ and $j\leq s\leq l$.

Let $\mathcal{P}$ be a finite collection of cells of $\mathbb{N}^{2}$. We call $\mathcal{P}$ $row \ convex$ if the horizontal cell interval $[A,B]$ is contained in $\mathcal{P}$ 
for any cells $A$ and $B$ of $\mathcal{P}$ whose lower left corners are in horizontal position. Similary one can define $column \ convex$. We call $\mathcal{P}$ $convex$ if it is row convex and column convex.

Let $\mathcal{P}$ be a finite collection of cells of $\mathbb{N}^{2}$ with $V(\mathcal{P})$ its vertex set. 
Let $S$ denote the polynomial ring over a field $K$ whose the variables are those $x_{a}$ with $a\in V(\mathcal{P})$. We say that an interval $[a,b]$ is
 $an \ interval \ of \mathcal{P}$ if $\mathcal{P}_{[a,b]}\subset \mathcal{P}$, where $\mathcal{P}_{[a,b]}$ is the interval $[a,b]$ regarded as a polyomino. For each interval $[a,b]$ of $\mathcal{P}$, 
we introduce the binomial $f_{a,b}=x_{a}x_{b}-x_{c}x_{d}$, where $c$ and $d$ are the anti-diagonals of $[a,b]$. Such a binomial $f_{a,b}$ is said to be an $inner$ 2-$minor$ of $\mathcal{P}$. 
Write $I_{\mathcal{P}}$ for the ideal generated by all inner 2-minors of $\mathcal{P}$. 
If $\mathcal{P}$ happens to be a polyomino, then $I_{\mathcal{P}}$ will  be called a $polyomino$ $ideal$. The $K$-algebra $S/I_{\mathcal{P}}$ is denoted $K[\mathcal{P}]$ and is called the coordinate ring of $\mathcal{P}$. 

The following result has been shown by Qureshi in $\cite{Q}$.

\begin{theorem}
Let $\mathcal{P}$ be a convex polyomino. Then $K[\mathcal{P}]$ is a normal Cohen-Macaulay domain of dimension $|V(\mathcal{P})|$ - $|\mathcal{P}|$.
\end{theorem}
We may assume that $V(\mathcal{P})\subset [(1,1),(m,n)]$. To $\mathcal{P}$ we attach the toric ring \[R=K[x_{i}y_{j} \ | \ (i,j)\in V(\mathcal{P})]\subset K[x_{1},\ldots, x_{m}, y_{1},\ldots, y_{n}].\]
For the sake of convenience, we denote for $a=(i,j)\in V(\mathcal{P})$ the variable $x_{a}$ in $S$ by $x_{ij}$. The polyomino ideal $I_{\mathcal{P}}$ may be viewed as follows as a toric ideals. 
Consider the $K$-algebra homomorphism $\phi: S\rightarrow R$ with $\phi(x_{ij})=x_{i}y_{j}$ for all $(i,j)\in V(\mathcal{P})$. It follows that $I_{\mathcal{P}}=Ker(\phi)$ and $K[\mathcal{P}]$ 
may be indentified with the toric ring $R$. It follows that $R$ is normal and we recall that by a well known result of Danilov and Stanley $(\cite{BG}, \cite{BH}, \cite{MS}, \cite{V} )$ the canonical module $\omega_{R}$ of
 $R,$ with respect to
standard grading, can be expressed as an ideal of $R$ generated by
 monomials, that is
$\omega_{R}=(\{x^{a}| \ a\in {\mathbb{N}}B\cap
 \relint({\mathbb{R_{+}}}B)\}),$ where  \[B:=\{
\log(x_{i}y_{j}) \ | \ (i,j)\in V(\mathcal{P})\}\subset \NN^{m+n}\] is the set of the exponents of the generators of the toric ring $R$.

\subsection{Simple paths with general bounderies}
We briefly recall the most general problem to count paths in a region that is bounded by nonliniar upper and lower boundaries (see $\cite{K}$ for more details). 

Let $a_{1}\leq a_{2}\leq \cdots \leq a_{n}$ and $b_{1}\leq b_{2}\leq \cdots \leq b_{n}$ be integers with $a_{i}\geq b_{i}.$ We abbreviate $\bf{a}$=$(a_{1},a_{2},\ldots,a_{n})$ and $\bf{b}$=$(b_{1},b_{2},\ldots,b_{n}).$ 
 By $L((0,b_{1})\rightarrow (n,a_{n}) \ | \  \bf{a} \geq \bf{y} \geq \bf{b})$  we denote the set of all lattice paths from $(0,b_{1})$ to $(n,a_{n})$ that satisfy the property that for all $i=1, 2, \dots, n$ the height $y_{i}$
of the $i-$th horizontal step is in the interval $[b_{i}, a_{i}].$ 
\begin{theorem} $( \cite{K},$ \textrm{Theorem  10.7.1.}$)$
Let $\bf{a}$=$(a_{1},a_{2},\ldots,a_{n})$ and $\bf{b}$=$(b_{1},b_{2},\ldots,b_{n})$ be integer sequences with $a_{1}\leq a_{2}\leq \cdots \leq a_{n},$ $b_{1}\leq b_{2}\leq \cdots \leq b_{n}$ and $a_{i}\geq b_{i},$ $i=1, 2, \dots, n.$
The number of all paths from $(0,b_{1})$ to $(n,a_{n})$ satisfying the property that for all $i=1, 2, \dots, n$ the height $y_{i}$
of the $i-$th horizontal step is between $b_{i}$ and $a_{i}$ is given by \[ | \ L((0,b_{1})\rightarrow (n,a_{n}) \ | \  {\bf{a} \geq \bf{y} \geq \bf{b})} \ |   = \det_{1\leq i,j \leq n} \left(  \binom{a_{i}-b_{j}+1}{j-i+1} \right) .\]
\end{theorem}

\section{The Main result}

\begin{definition}
For $n, t, p \in \mathbb{N}$, $1\leq t < n$, $p\geq 1$ we define $ \begin{bmatrix}  n \\ t \end{bmatrix}_{p} $  to be the number of elements of $A(n,t,p)$, where
\[A(n,t,p):=\{\alpha \in \ZZ_>^{pt+1} \ | \ \sum_{i=1}^{pt+1} \alpha_{i} = p(n-t), \sum_{i=1}^{kt} \alpha_{i} \leq k(n-t) \ \textrm{for} \ \textrm{all} \ k \in [p-1]\}.\]
\end{definition}

\begin{proposition}
For any $n, t, p \in \mathbb{N}$, $1\leq t < n$, $p\geq 1$ we have:\\
$1)$ \(
\begin{bmatrix}  n \\ t \end{bmatrix}_{p} = \begin{bmatrix}  n \\ n-t \end{bmatrix}_{p}.
\)\\
$2)$
\(\begin{bmatrix}  n \\ 1 \end{bmatrix}_{p}=\begin{bmatrix}  n \\ n-1 \end{bmatrix}_{p}=C_{p+1}(n), \textrm{where} \ C_{p}(n)=\frac{1}{(n-1)p+1}\binom{np}{p} \ \textrm{are} \ \textrm{the} \ 
\textrm{Fuss-Catalan} \ \textrm{numbers}.\)
\end{proposition}
\begin{proof}
Let $n, t, p \in \mathbb{N}$, $1\leq t < n$, $p\geq 1$ be fixed. If we examine the elements of $A(n,t,p)$ then we see that we can get rid of
$\alpha_{pt+1}$ at the expense of getting another inequality: so the number $ \begin{bmatrix}  n \\ t \end{bmatrix}_{p} $  is the same as the number of all vectors

 \[(\alpha_{1}, \ldots, \alpha_{pt})\in \ZZ_>^{pt} \ \textrm{with} \
\sum_{i=1}^{kt} \alpha_{i} \leq k(n-t) \ \textrm{for} \ \textrm{all} \ k \in [p]. \ \ \ \ \  \ \ \ \ \  \  \  (1) \]
These vectors can be better represented by lattice paths from $(0,0)$ to $(pt,p(n-t))$ consisting of unit horizontal and vertical steps with the constraint that is imposed by the inequalities $(1).$
More precisely, between the vertical lines $x=(k-1)t$ and $x=kt$ the path must stay weakly below height $k(n-t)$ for any $k \in [p].$
If we denote by $\bf{a}$=$(a_{1},a_{2},\ldots,a_{pt})$ and $\bf{b}$=$(b_{1},b_{2},\ldots,b_{pt})$ where $a_{1}=\cdots =a_{t}=n-t,$ $a_{t+1}=\cdots =a_{2t}=2(n-t),$ \ldots, $a_{(p-1)t+1}=\cdots =a_{pt}=p(n-t)$ and $b_{1}=\cdots=b_{pt}=0$ 
 using $( \cite{K},$ \textrm{Theorem  10.7.1.}$)$ we have \[\begin{bmatrix}  n \\ t \end{bmatrix}_{p}=| \ L((0,0)\rightarrow (pt,a_{pt}) \ | \  {\bf{a} \geq \bf{y} \geq \bf{b})} \ |   = \det_{1\leq i,j \leq pt} \left(  \binom{a_{i}+1}{j-i+1} \right) .\]
By using similar argument, if we denote by $\bf{a'}$=$(a_{1}',a_{2}',\ldots,a_{p(n-t)}')$ and $\bf{b'}$=$(b_{1}',b_{2}',\ldots,b_{p(n-t)}')$ where 
$a_{1}'=\cdots =a_{n-t}'=t,$ $a_{n-t+1}'=\cdots =a_{2(n-t)}'=2t,$ \ldots, $a_{(p-1)(n-t)+1}'=\cdots =a_{p(n-t)}'=pt$ and $b_{1}'=\cdots=b_{p(n-t)}'=0$ 
we have \[\begin{bmatrix}  n \\ n-t \end{bmatrix}_{p}=| \ L((0,0)\rightarrow (p(n-t),a_{pt}') \ | \  {\bf{a'} \geq \bf{y} \geq \bf{b'})} \ |   = \det_{1\leq i,j \leq p(n-t)} \left(  \binom{a_{i}'+1}{j-i+1} \right) .\]
Since any path of $L((0,0)\rightarrow (p(n-t),a_{pt}')$ is a reflection of a path from $L((0,0)\rightarrow (pt,a_{pt})$ in a line parallel to the antidiagonal $x+y=0$ we get \(
\begin{bmatrix}  n \\ t \end{bmatrix}_{p} = \begin{bmatrix}  n \\ n-t \end{bmatrix}_{p}.\)\\ For $t=1$ one obtains \(\begin{bmatrix}  n \\ 1 \end{bmatrix}_{p}=C_{p+1}(n), \) 
Fuss-Catalan numbers since the above lattice path model reduces in that case to the standard lattice path model for Fuss-Catalan numbers.
\end{proof}

\begin{corollary}
For any $n, t, p \in \mathbb{N}$, $1\leq t < n$, $p\geq 1$ we have:\\
$1)$ For $a_{1}=\cdots =a_{t}=n-t,$ $a_{t+1}=\cdots =a_{2t}=2(n-t),$ \ldots, $a_{(p-1)t+1}=\cdots =a_{pt}=p(n-t)$ 
we have \[\begin{bmatrix}  n \\ t \end{bmatrix}_{p} = \det_{1\leq i,j \leq pt} \left(  \binom{a_{i}+1}{j-i+1} \right) .\]
$2)$ For $a_{1}'=\cdots =a_{n-t}'=t,$ $a_{n-t+1}'=\cdots =a_{2(n-t)}'=2t,$ \ldots, $a_{(p-1)(n-t)+1}'=\cdots =a_{p(n-t)}'=pt$  
we have \[\begin{bmatrix}  n \\ n-t \end{bmatrix}_{p}= \det_{1\leq i,j \leq p(n-t)} \left(  \binom{a_{i}'+1}{j-i+1} \right) .\]
\end{corollary}
\begin{proof}
See the proof above.
\end{proof}

\begin{example}
For $n=3$, $t=1$ and $p=3$ we have:\\
$i)$ $\bf{a}$=$(2,4,6)$,  \[\begin{bmatrix}  3 \\ 1 \end{bmatrix}_{3}= \det_{1\leq i,j \leq 3} \left(  \binom{a_{i}+1}{j-i+1} \right)= \begin{vmatrix}
3 & 3 & 1 \\ 
1 & 5 & 10 \\ 
0 & 1 & 7  \notag
\end{vmatrix} =55=C_{4}(3).\]
$ii)$ $\bf{a}'$=$(1,1,2,2,3,3)$,  \[\begin{bmatrix}  3 \\ 2 \end{bmatrix}_{3}= \det_{1\leq i,j \leq 6} \left(  \binom{a_{i}'+1}{j-i+1} \right)= \begin{vmatrix}
2 & 1 & 0 & 0 & 0 & 0 \\ 
1 & 2 & 1 & 0 & 0 & 0 \\ 
0 & 1 & 3 & 3 & 1 & 0 \\ 
0 & 0 & 1 & 3 & 3 & 1 \\ 
0 & 0 & 0 & 1 & 4 & 6 \\ 
0 & 0 & 0 & 0 & 1 & 4  \notag
\end{vmatrix} =55.\]
\end{example}

Let $\mathcal{P}=\mathcal{P}
  \left( {\begin{array}{cccc}
   u_{1} & u_{2} & . . . & u_{p} \\
   r_{1} & r_{2} & . . . & r_{p} \\
  \end{array} } \right)$ 
be the polyomino of Figure 1. The coordinate ring $K[\mathcal{P}]$ is a normal Cohen-Macaulay domain and it may be identified with the toric ring $R=K[x_{i}y_{j} \ | \ (i,j)\in V(\mathcal{P})]$.
We will describe the irreducible representation of the cone generated by \[B=\{\log(x_{i}y_{j}) \ | \ (i,j)\in V(\mathcal{P})\}\subset \NN^{B_{p}+A_{p}},\] 
the set of the exponents of the generators of the toric ring $R$.\\
Let \[\nu_{i}=-\sum_{k\in [B_{i}-1]}e_{k}+\sum_{k\in [A_{i}]}e_{B_{p}+k}, \  \ \nu=\sum_{k\in [B_{p}]}e_{k}-\sum_{k\in [A_{p}]}e_{B_{p}+k}, \] for $i\in [p-1]$, where $\{e_{k}\}_{1\leq k\leq A_{p}+B_{p}}$ is the canonical base of $\RR^{A_{p}+B_{p}}$.
\begin{lemma}
The cone generated by $B=\{\log(x_{i}y_{j}) \ | \ (i,j)\in V(\mathcal{P})\}\subset \NN^{B_{p}+A_{p}}$, the set of the exponents of the generators of the toric ring $R$, 
has irreducible representation \[\RR_{+}B=(\bigcap_{a\in N} H_{a}^{+})\cap H_{\nu},\]
where $N=\{\nu_{i}, \ e_{k} \ | \ i\in[p-1], \ k\in[A_{p}+B_{p}]\}$.
\end{lemma}
\begin{proof}
The set of the exponents of the generators of the toric ring $R$ consists of the vectors $g_{i,j}^{1}=e_{i}+e_{B_{p}+j}$ for $i\in[B_{1}-1], \ j\in[A_{1}],$ 
and $g_{i,j}^{t}=e_{i}+e_{B_{p}+j}$ for $2\leq t \leq p, \ B_{t-1}\leq i \leq B_{t}-1, \ j\in[A_{t}]$. \\
Now we will prove that $\RR_{+}B\subseteq H_{a}^{+}\cap H_{\nu}$ 
for any $a\in N$. It is clear that $\RR_{+}B\subseteq H_{\nu}$ and 
$\RR_{+}B\subseteq H_{e_{k}}^{+}$ for all $k\in[A_{p}+B_{p}]$.
If $i\in[B_{1}-1]$ and $j\in [A_{1}]$ we have $g_{i,j}^{1} \in H_{\nu_{s}}$ for any $s\in [p-1]$. If $2\leq t \leq p$, $B_{t-1}\leq i \leq B_{t}-1$, $j\in [A_{t}]$ 
then $g_{i,j}^{t} \in H_{\nu_{s}}^{+}$ for any $1\leq s \leq t-1$ and $g_{i,j}^{t} \in H_{\nu_{s}}$ for any $t\leq s \leq p$. 
Thus $\RR_{+}B\subseteq(\bigcap_{a\in N} H_{a}^{+})\cap H_{\nu}$.\\
Now we will prove the converse inclusion 
$\RR_{+}B\supseteq (\bigcap_{a\in N}H^{+}_{a})\cap H_{\nu}.$\\ It is  
enough to prove that the extremal rays of the cone 
$\bigcap_{a\in N}H^{+}_{a}$ \ are in
${\RR_{+}}B.$ \ Any extremal ray of the cone  
$\bigcap_{a\in N}H^{+}_{a}$ \ can be written  as the intersection of  
$A_{p}+B_{p}-2$ hyperplanes $H_{a}$,  with $a\in N.$ We will prove that all the exponents of the generators of the toric ring $R$ 
are the extremal rays of the cone $\bigcap_{a\in N}H^{+}_{a}$. There are the following posibilities to obtain extremal rays by intersection of $A_{p}+B_{p}-2$ hyperplanes.

$1.$ For $i\in [B_{1}-1]$ and $j\in [A_{1}]$,  let $k_{1}< \ldots < k_{A_{p}+B_{p}-3}$ be the sequence of integers such that $\{k_{1}, \ldots , k_{A_{p}+B_{p}-3}\}\subset [A_{p}+B_{p}]\setminus \{i, B_{p}+j\}$.\\
The system of equations :\\
$ \ (*) \ \begin{cases}
z_{k_{1}}=0,\\
z_{k_{2}}=0,\\
\vdots \\
z_{k_{A_{p}+B_{p}-3}}=0,\\
H_{\nu_{1}}(z)=0.
\end{cases}$ \\
admits the solution $g_{i,j}^{1}.$

$2.$ For $B_{t-1}\leq i\leq B_{t}-1$ and $j\in [A_{t}]$ with $2\leq t \leq p-1$,  let $k_{1}< \ldots < k_{A_{p}+B_{p}-3}$ be the sequence of integers such that $\{k_{1}, \ldots , k_{A_{p}+B_{p}-3}\}\subset [A_{p}+B_{p}]\setminus \{i, B_{p}+j\}$. \\The system of equations :\\
$ \ (**) \ \begin{cases}
z_{k_{1}}=0,\\
z_{k_{2}}=0,\\
\vdots \\
z_{k_{A_{p}+B_{p}-3}}=0,\\
H_{\nu_{t}}(z)=0.
\end{cases}$\\
admits the solution $g_{i,j}^{t}$.

$3.$ For $B_{p-1}\leq i\leq B_{p}$ and $j\in [A_{p}]$,  let $k_{1}< \ldots < k_{A_{p}+B_{p}-2}$ be the sequence of integers such that $\{k_{1}, \ldots , k_{A_{p}+B_{p}-2}\}= [A_{p}+B_{p}]\setminus \{i, B_{p}+j\}$.\\
The system of equations :\\
$ \ (***) \ \begin{cases}
z_{k_{1}}=0,\\
z_{k_{2}}=0,\\
\vdots \\
z_{k_{A_{p}+B_{p}-2}}=0.
\end{cases}$\\
admits the solution $g_{i,j}^{p}$.\\
Thus $\RR_{+}B\supseteq(\bigcap_{a\in N} H_{a}^{+})\cap H_{\nu}$.
\end{proof}
\begin{rem}
If $G=(V(G), E(G))$ is a finite connected graph on the vertex set $V(G)=[n], n\geq 2$ and $e=\{i,j\}\in E(G)$ is an edge then we define $\rho(e) \in \RR^{n}$ by $\rho(e)=e_{i}+e_{j}$ where $e_{i}$ is the $i$th unit coordinate vector in $\RR^n$.
We write $\mathcal{P}_{G}\subset \RR^{n}$ for the convex hull of the finite set $\{\rho(e) \ | \ e \in E(G)\} \subset \RR^{n}$ and we call $\mathcal{P}_{G}$ the \emph{edge polytope} of $G.$  In $($\cite{OH},  \textrm{Theorem  1.7}$)$
 Ohsugi and Hibi described the set of facets of the edge polytope $\mathcal{P}_{G}$. Lemma 7 can be derived from this result and using $( \cite{Z},$ \textrm{Theorems 1.1, 1.2, 1.3.}$)$.
\end{rem}
\begin{example} 
We present two polyominoes $\mathcal{P}_{1}, \mathcal{P}_{2}$ whose Cohen-Macaulay types of coordinate rings are $r(K[\mathcal{P}_{1}])=\begin{bmatrix}  3 \\ 1 \end{bmatrix}_{3}$ and
$r(K[\mathcal{P}_{2}])=\begin{bmatrix}  3 \\ 2 \end{bmatrix}_{3}$.
\end{example}
$i)$ Consider the polyomino $\mathcal{P}_{1}=\mathcal{P}
  \left( {\begin{array}{cccc}
   3 & 3 & 3 \\
   1 & 1 & 1 \\
  \end{array} } \right)$. We have $A_{1}=4, \ A_{2}=7, \ A_{3}=10$ and $B_{1}=2, \ B_{2}=3, \ B_{3}=4$. The quotient ring by the polyomino ideal $I_{\mathcal{P}_{1}}$ is isomorphic to the edge subring 
	\[R_{1}=K[x_{1}y_{1},\ldots, x_{1}y_{4}, x_{2}y_{1},\ldots, x_{2}y_{7}, x_{3}y_{1}, \ldots, x_{3}y_{10}, x_{4}y_{1}, \ldots, x_{4}y_{10}].\]
	Computation with NORMALIZ, we get the Hilbert series of $R_{1}$ \[H_{R_{1}}(t)=\frac{1+18t+66t^2+55t^3}{(1-t)^{13}}\]
	and the generators of the canonical module of $R_{1}$
	\begin{center}
\begin{tabular}{c c c c c c}
  $\omega_{R_{1}}$=($x_{1}x_{2}x_{3}x_{4}^{7}y$, & $x_{1}x_{2}x_{3}^{2}x_{4}^{6}y$, & $x_{1}x_{2}x_{3}^{3}x_{4}^{5}y$, & $x_{1}x_{2}x_{3}^{4}x_{4}^{4}y$, & $x_{1}x_{2}x_{3}^{5}x_{4}^{3}y$, & $x_{1}x_{2}x_{3}^{6}x_{4}^{2}y$,   \\
	\ \ \ \ \ \ \ \ \ $x_{1}x_{2}x_{3}^{7}x_{4}y$, & $x_{1}x_{2}^{2}x_{3}x_{4}^{6}y$, & $x_{1}x_{2}^{2}x_{3}^{2}x_{4}^{5}y$, & $x_{1}x_{2}^{2}x_{3}^{3}x_{4}^{4}y$, & $x_{1}x_{2}^{2}x_{3}^{4}x_{4}^{3}y$, & $x_{1}x_{2}^{2}x_{3}^{5}x_{4}^{2}y$,  \\ 
	\ \ \ \ \ \ \ \ \ $x_{1}x_{2}^{2}x_{3}^{6}x_{4}y$, & $x_{1}x_{2}^{3}x_{3}x_{4}^{5}y$, & $x_{1}x_{2}^{3}x_{3}^{2}x_{4}^{4}y$, & $x_{1}x_{2}^{3}x_{3}^{3}x_{4}^{3}y$, & $x_{1}x_{2}^{3}x_{3}^{4}x_{4}^{2}y$, & $x_{1}x_{2}^{3}x_{3}^{5}x_{4}y$,   \\  
	\ \ \ \ \ \ \ \ \ $x_{1}x_{2}^{4}x_{3}x_{4}^{4}y$, & $x_{1}x_{2}^{4}x_{3}^{2}x_{4}^{3}y$, & $x_{1}x_{2}^{4}x_{3}^{3}x_{4}^{2}y$, & $x_{1}x_{2}^{4}x_{3}^{4}x_{4}y$, & $x_{1}x_{2}^{5}x_{3}x_{4}^{3}y$, & $x_{1}x_{2}^{5}x_{3}^{2}x_{4}^{2}y$, \\
	\ \ \ \ \ \ \ \ \ $x_{1}x_{2}^{5}x_{3}^{3}x_{4}y$, & $x_{1}^{2}x_{2}x_{3}x_{4}^{6}y$, & $x_{1}^{2}x_{2}x_{3}^{2}x_{4}^{5}y$, & $x_{1}^{2}x_{2}x_{3}^{3}x_{4}^{4}y$, & $x_{1}^{2}x_{2}x_{3}^{4}x_{4}^{3}y$, & $x_{1}^{2}x_{2}x_{3}^{5}x_{4}^{2}y$, \\
	\ \ \ \ \ \ \ \ \ $x_{1}^{2}x_{2}x_{3}^{6}x_{4}y$, & $x_{1}^{2}x_{2}^{2}x_{3}x_{4}^{5}y$, & $x_{1}^{2}x_{2}^{2}x_{3}^{2}x_{4}^{4}y$, & $x_{1}^{2}x_{2}^{2}x_{3}^{3}x_{4}^{3}y$, & $x_{1}^{2}x_{2}^{2}x_{3}^{4}x_{4}^{2}y$, & $x_{1}^{2}x_{2}^{2}x_{3}^{5}x_{4}y$, \\
\ \ \ \ \ \ \ \ \ $x_{1}^{2}x_{2}^{3}x_{3}x_{4}^{4}y$, & $x_{1}^{2}x_{2}^{3}x_{3}^{2}x_{4}^{3}y$, & $x_{1}^{2}x_{2}^{3}x_{3}^{3}x_{4}^{2}y$, & $x_{1}^{2}x_{2}^{3}x_{3}^{4}x_{4}y$, & $x_{1}^{2}x_{2}^{4}x_{3}x_{4}^{3}y$, & $x_{1}^{2}x_{2}^{4}x_{3}^{2}x_{4}^{2}y$, \\
\ \ \ \ \ \ \ \ \ $x_{1}^{2}x_{2}^{4}x_{3}^{3}x_{4}y$, & $x_{1}^{3}x_{2}x_{3}x_{4}^{5}y$, & $x_{1}^{3}x_{2}x_{3}^{2}x_{4}^{4}y$, & $x_{1}^{3}x_{2}x_{3}^{3}x_{4}^{3}y$, & $x_{1}^{3}x_{2}x_{3}^{4}x_{4}^{2}y$, & $x_{1}^{3}x_{2}x_{3}^{5}x_{4}y$, \\
\ \ \ \ \ \ \ \ \ $x_{1}^{3}x_{2}^{2}x_{3}x_{4}^{4}y$, & $x_{1}^{3}x_{2}^{2}x_{3}^{2}x_{4}^{3}y$, & $x_{1}^{3}x_{2}^{2}x_{3}^{3}x_{4}^{2}y$, & $x_{1}^{3}x_{2}^{2}x_{3}^{4}x_{4}y$, & $x_{1}^{3}x_{2}^{3}x_{3}x_{4}^{3}y$, & $x_{1}^{3}x_{2}^{3}x_{3}^{2}x_{4}^{2}y$,\\
\ \ \ \ \ \ \ \ \ \ \ \ \ \ $x_{1}^{3}x_{2}^{3}x_{3}^{3}x_{4}y)R_{1}, $ 
\end{tabular}
\end{center}  
where $y=y_{1}\cdot \ldots \cdot y_{10}.$ \\
 Hence, we have the Cohen-Macaulay type of $R_{1}$ is 55$=\begin{bmatrix}  3 \\ 1 \end{bmatrix}_{3}$.

$ii)$ Consider the polyomino $\mathcal{P}_{2}=\mathcal{P}
  \left( {\begin{array}{cccc}
   3 & 3 & 3 \\
   2 & 2 & 2 \\
  \end{array} } \right).$ We have $A_{1}=4, \ A_{2}=7, \ A_{3}=10$ and $B_{1}=3, \ B_{2}=5, \ B_{3}=7$. 
	The quotient ring by the polyomino ideal $I_{\mathcal{P}_{2}}$ is isomorphic to the edge subring 
	\[R_{2}=K[x_{1}y_{1},\ldots, x_{1}y_{4}, x_{2}y_{1},\ldots, x_{2}y_{4}, x_{3}y_{1},\ldots, x_{3}y_{7}, x_{4}y_{1},\ldots, x_{4}y_{7}, \]
	\[x_{5}y_{1}, \ldots, x_{5}y_{10}, x_{6}y_{1}, \ldots, x_{6}y_{10}, x_{7}y_{1}, \ldots, x_{7}y_{10}]. \ \ \ \ \]
		Computation with NORMALIZ, we get the Hilbert series of $R_{2}$ \[H_{R_{2}}(t)=\frac{1+36t+318t^2+960t^3+1071t^4+444t^5+55t^6}{(1-t)^{16}}\]
	and the generators of the canonical module of $R_{2}$
\begin{center}
\begin{tabular}{c c c c}
  $\omega_{R_{2}}$=($x_{1}x_{2}x_{3}x_{4}x_{5}x_{6}x_{7}^{4}y$, & $x_{1}x_{2}x_{3}x_{4}x_{5}x_{6}^{2}x_{7}^{3}y$, & $x_{1}x_{2}x_{3}x_{4}x_{5}x_{6}^{3}x_{7}^{2}y$, & $x_{1}x_{2}x_{3}x_{4}x_{5}x_{6}^{4}x_{7}y$, \\
	\ \ \ \ \ \ \ \ \ $x_{1}x_{2}x_{3}x_{4}x_{5}^{2}x_{6}x_{7}^{3}y$, & $x_{1}x_{2}x_{3}x_{4}x_{5}^{2}x_{6}^{2}x_{7}^{2}y$, & $x_{1}x_{2}x_{3}x_{4}x_{5}^{2}x_{6}^{3}x_{7}y$, & $x_{1}x_{2}x_{3}x_{4}x_{5}^{3}x_{6}x_{7}^{2}y$, \\
	\ \ \ \ \ \ \ \ \ $x_{1}x_{2}x_{3}x_{4}x_{5}^{3}x_{6}^{2}x_{7}y$, & $x_{1}x_{2}x_{3}x_{4}x_{5}^{4}x_{6}x_{7}y$, & $x_{1}x_{2}x_{3}x_{4}^{2}x_{5}x_{6}x_{7}^{3}y$, & $x_{1}x_{2}x_{3}x_{4}^{2}x_{5}x_{6}^{2}x_{7}^{2}y$, \\
	\ \ \ \ \ \ \ \ \ $x_{1}x_{2}x_{3}x_{4}^{2}x_{5}x_{6}^{3}x_{7}y$, & $x_{1}x_{2}x_{3}x_{4}^{2}x_{5}^{2}x_{6}x_{7}^{2}y$, & $x_{1}x_{2}x_{3}x_{4}^{2}x_{5}^{2}x_{6}^{2}x_{7}y$, & $x_{1}x_{2}x_{3}x_{4}^{2}x_{5}^{3}x_{6}x_{7}y$, \\
	\ \ \ \ \ \ \ \ \ $x_{1}x_{2}x_{3}x_{4}^{3}x_{5}x_{6}x_{7}^{2}y$, & $x_{1}x_{2}x_{3}x_{4}^{3}x_{5}x_{6}^{2}x_{7}y$, & $x_{1}x_{2}x_{3}x_{4}^{3}x_{5}^{2}x_{6}x_{7}y$, & $x_{1}x_{2}x_{3}^{2}x_{4}x_{5}x_{6}x_{7}^{3}y$, \\
	\ \ \ \ \ \ \ \ \ $x_{1}x_{2}x_{3}^{2}x_{4}x_{5}x_{6}^{2}x_{7}^{2}y$, & $x_{1}x_{2}x_{3}^{2}x_{4}x_{5}x_{6}^{3}x_{7}y$, & $x_{1}x_{2}x_{3}^{2}x_{4}x_{5}^{2}x_{6}x_{7}^{2}y$, & $x_{1}x_{2}x_{3}^{2}x_{4}x_{5}^{2}x_{6}^{2}x_{7}y$, \\
	\ \ \ \ \ \ \ \ \ $x_{1}x_{2}x_{3}^{2}x_{4}x_{5}^{3}x_{6}x_{7}y$, & $x_{1}x_{2}x_{3}^{2}x_{4}^{2}x_{5}x_{6}x_{7}^{2}y$, & $x_{1}x_{2}x_{3}^{2}x_{4}^{2}x_{5}x_{6}^{2}x_{7}y$, & $x_{1}x_{2}x_{3}^{2}x_{4}^{2}x_{5}^{2}x_{6}x_{7}y$, \\
	\ \ \ \ \ \ \ \ \ $x_{1}x_{2}x_{3}^{3}x_{4}x_{5}x_{6}x_{7}^{2}y$, & $x_{1}x_{2}x_{3}^{3}x_{4}x_{5}x_{6}^{2}x_{7}y$, & $x_{1}x_{2}x_{3}^{3}x_{4}x_{5}^{2}x_{6}x_{7}y$, & $x_{1}x_{2}^{2}x_{3}x_{4}x_{5}x_{6}x_{7}^{3}y$, \\
	\ \ \ \ \ \ \ \ \ $x_{1}x_{2}^{2}x_{3}x_{4}x_{5}x_{6}^{2}x_{7}^{2}y$, & $x_{1}x_{2}^{2}x_{3}x_{4}x_{5}x_{6}^{3}x_{7}y$, & $x_{1}x_{2}^{2}x_{3}x_{4}x_{5}^{2}x_{6}x_{7}^{2}y$, & $x_{1}x_{2}^{2}x_{3}x_{4}x_{5}^{2}x_{6}^{2}x_{7}y$, \\
	\ \ \ \ \ \ \ \ \ $x_{1}x_{2}^{2}x_{3}x_{4}x_{5}^{3}x_{6}x_{7}y$, & $x_{1}x_{2}^{2}x_{3}x_{4}^{2}x_{5}x_{6}x_{7}^{2}y$, & $x_{1}x_{2}^{2}x_{3}x_{4}^{2}x_{5}x_{6}^{2}x_{7}y$, & $x_{1}x_{2}^{2}x_{3}x_{4}^{2}x_{5}^{2}x_{6}x_{7}y$, \\
	\ \ \ \ \ \ \ \ \ $x_{1}x_{2}^{2}x_{3}^{2}x_{4}x_{5}x_{6}x_{7}^{2}y$, & $x_{1}x_{2}^{2}x_{3}^{2}x_{4}x_{5}x_{6}^{2}x_{7}y$, & $x_{1}x_{2}^{2}x_{3}^{2}x_{4}x_{5}^{2}x_{6}x_{7}y$, & $x_{1}^{2}x_{2}x_{3}x_{4}x_{5}x_{6}x_{7}^{3}y$, \\
	\ \ \ \ \ \ \ \ \ $x_{1}^{2}x_{2}x_{3}x_{4}x_{5}x_{6}^{2}x_{7}^{2}y$, & $x_{1}^{2}x_{2}x_{3}x_{4}x_{5}x_{6}^{3}x_{7}y$, & $x_{1}^{2}x_{2}x_{3}x_{4}x_{5}^{2}x_{6}x_{7}^{2}y$, & $x_{1}^{2}x_{2}x_{3}x_{4}x_{5}^{2}x_{6}^{2}x_{7}y$, \\
	\ \ \ \ \ \ \ \ \ $x_{1}^{2}x_{2}x_{3}x_{4}x_{5}^{3}x_{6}x_{7}y$, & $x_{1}^{2}x_{2}x_{3}x_{4}^{2}x_{5}x_{6}x_{7}^{2}y$, & $x_{1}^{2}x_{2}x_{3}x_{4}^{2}x_{5}x_{6}^{2}x_{7}y$, & $x_{1}^{2}x_{2}x_{3}x_{4}^{2}x_{5}^{2}x_{6}x_{7}y$, \\
	\ \ \ \ \ \ \ \ \ $x_{1}^{2}x_{2}x_{3}^{2}x_{4}x_{5}x_{6}x_{7}^{2}y$, & $x_{1}^{2}x_{2}x_{3}^{2}x_{4}x_{5}x_{6}^{2}x_{7}y$, & \ \ \ \ \ $x_{1}^{2}x_{2}x_{3}^{2}x_{4}x_{5}^{2}x_{6}x_{7}y)R_{2},$
\end{tabular}
\end{center}  
where $y=y_{1}\cdot \ldots \cdot y_{10}.$ \\
 Hence, we have the Cohen-Macaulay type of $R_{2}$ is 55$=\begin{bmatrix}  3 \\ 2 \end{bmatrix}_{3}$.	
	
In general it is almost impossible to describe the generators of the canonical module of the coordinate ring for a polyomino. 
Our final result gives a description in a particular case.\\

For $n, t, p \in \mathbb{N}$, $1\leq t < n$, $p\geq 1$, let $u_{1}=\cdots=u_{p}=n$ and  $r_{1}=\cdots=r_{p}=t$ and we consider 
$\mathcal{P}=\mathcal{P}
  \left( {\begin{array}{cccc}
   u_{1} & u_{2} & . . . & u_{p} \\
   r_{1} & r_{2} & . . . & r_{p} \\
  \end{array} } \right)$.
	
	\begin{theorem}
	Let $S=K[x_{1},\ldots, x_{pt+1}, y_{1},\ldots, y_{pn+1}]$ be a standard graded polynomial ring over a field $K$ and 
	$R=K[x_{i}y_{j} \ | \ (i,j)\in V(\mathcal{P})].$ Then the Cohen-Macaulay type of $R$ is $r(R)=\begin{bmatrix}  n \\ t \end{bmatrix}_{p} $.
	\end{theorem}
	
	\begin{proof}
	We know that $K[\mathcal{P}]$ is a normal Cohen-Macaulay domain of dimension $|V(\mathcal{P})|$ - $|\mathcal{P}|=pt+pn+1$. 
It follows that $R$ is normal and we recall that by a well known formula of Danilov-Stanley the canonical module $\omega_{R}$ of
 $R,$ with respect to standard grading, can be expressed as an ideal of $R$ generated by
 monomials, that is $\omega_{R}=(\{x^{a}| \ a\in {\mathbb{N}}B\cap \relint({\mathbb{R_{+}}}B)\}),$ 
where  \[B:=\{\log(x_{i}y_{j}) \ | \ (i,j)\in V(\mathcal{P})\}\subset \NN^{pt+pn+2}\] is the set of the exponents of the generators of the toric ring $R$.
Thus the canonical module is the ideal 
\[\omega_{R}=(\{x^{\alpha}y \ | \ \alpha\in \ZZ_>^{pt+1}, \  \sum_{i=1}^{pt+1}\alpha_{i}=pn+1, \ \sum_{i=1}^{kt}\alpha_{i} < kn+1 \ for \ all \ k \in [p-1]\})R,\]
where $x^{\alpha}=x_{1}^{\alpha_{1}}\cdot \ldots \cdot x_{pt+1}^{\alpha_{pt+1}}$ and $y=y_{1}\cdot \ldots \cdot y_{pn+1}$. \\
Then \[\omega_{R}=(\{x^{\alpha}y \ | \ \alpha - {\bf{1}} \in A(n,t,p)\})R,\] where \ ${\bf{1}}=(1,1,\ \ldots, 1) \in \ZZ_>^{pt+1}$ and so the Cohen-Macaulay type of $R$ is $r(R)=\begin{bmatrix}  n \\ t \end{bmatrix}_{p} $.
	\end{proof}

For $n, t, p \in \mathbb{N}$, $1\leq t < n$, $p\geq 1$, let $u_{1}=\cdots=u_{p}=n$ and  $r_{1}=\cdots=r_{p}=n-t$ and we consider 
$\mathcal{P}=\mathcal{P}
  \left( {\begin{array}{cccc}
   u_{1} & u_{2} & . . . & u_{p} \\
   r_{1} & r_{2} & . . . & r_{p} \\
  \end{array} } \right)$.
	
	\begin{corollary}
	Let $S=K[x_{1},\ldots, x_{p(n-t)+1}, y_{1},\ldots, y_{pn+1}]$ be a standard graded polynomial ring over a field $K$ and 
	$R=K[x_{i}y_{j} \ | \ (i,j)\in V(\mathcal{P})].$ Then the Cohen-Macaulay type of $R$ is $r(R)=\begin{bmatrix}  n \\ t \end{bmatrix}_{p} $.
	\end{corollary}

$\bf{Acknowledgement.}$	The author would like to thank the anonymous referee for her/his valuable suggestions.

Alin \c{S}tefan\\  Department of Computer Science, Information Technology, Mathematics and Physics\\ 
Petroleum and Gas University of Ploie\c{s}ti\\ 
B-dul. Bucure\c{s}ti, No. 39, Ploie\c{s}ti 100680, Romania\\
{\tt nastefan@upg-ploiesti.ro, alin11\_alg@yahoo.com }\vspace{3mm}

\end{document}